\documentclass{article}
\usepackage{amsmath, amssymb, amsthm, enumerate, thmtools, hyperref, mathtools, geometry, cite}

\newtheorem{theorem}{Theorem}
\newtheorem{lemma}[theorem]{Lemma}
\newtheorem{prop}[theorem]{Proposition}
\newtheorem{cor}[theorem]{Corollary}
\newtheorem{definition}[theorem]{Definition}

\DeclareMathOperator{\IP}{IP}
\DeclareMathOperator{\rank}{rank}
\DeclareMathOperator{\corank}{corank}
\DeclareMathOperator{\SPAN}{span}

\newcommand{\Z}{\mathbb{Z}}
\newcommand{\N}{\mathbb{N}}
\newcommand{\F}{\mathbb{F}}

\newcommand{\pDe}[1]{\prescript{\partial}{}{\varepsilon_{#1}}}
\newcommand{\T}[1]{\prescript{\partial}{}{w_{#1}}}

\newenvironment{amatrix}[1]{%
  \left[\begin{array}{@{}*{#1}{c}|c@{}}
}{%
  \end{array}\right]
}

\title{A recursion for the twist polynomial of a one-point join of normal binary delta-matroids}
\author{Charlton Li \vspace{1mm} \\
\small \textit{The Ohio State University} \\
\small \textit{Columbus, Ohio, USA}}
\date{}

\begin{document}

\maketitle

\noindent
\rule{\linewidth}{.5pt}

\noindent
\textbf{Abstract}

\noindent
The partial-dual Euler-genus polynomial was defined by Gross, Mansour, and Tucker to analyze how the Euler genus of a ribbon graph changes under partial duality, a generalization of geometric duality introduced by Chmutov. The twist polynomial defined by Yan and Jin extends the partial-dual Euler-genus polynomial to a polynomial on delta-matroids. We derive a recursion formula for the twist polynomial of a one-point join of looped simple graphs---equivalently, normal, binary delta-matroids. Our recursion applies to the partial-dual Euler-genus polynomial as a special case, where it generalizes a recursion obtained by Yan and Jin. We obtain relations for the twist polynomial on looped simple graphs evaluated at $-1/2$ and for the twist polynomial of a graph with a single looped vertex. A characterization is given for the feasible sets of the delta-matroid associated to a one-point join of looped simple graphs. We show that Yan and Jin's recursion extends to the twist polynomial on delta-matroids. \\



\noindent
\rule{\linewidth}{.5pt}

\section{Introduction}
Given a ribbon graph $\Gamma$, let $V(\Gamma)$ and $E(\Gamma)$ be the collection of vertices and edges of $\Gamma$ respectively. Chmutov \cite{Chmutov_2009} introduced the \textit{partial dual} $\Gamma^A$ with respect to an edge subset $A\subseteq E(\Gamma)$ as a generalization of geometric duality in order to unify various results that expressed the Jones polynomial as an evaluation of the Bollob\'{a}s-Riordan polynomial. The Euler genus $\varepsilon(\Gamma)$ of a ribbon graph is invariant under geometric duality, but not invariant under partial duality. To analyze how partial duality changes the Euler genus, Gross, Mansour, and Tucker defined the \textit{partial-dual Euler-genus polynomial} of $\Gamma$ in \cite{Gross_Mansour_Tucker_2020} as  
        \[\pDe{\Gamma}(z)=\sum_{A\subseteq E(\Gamma)}z^{\varepsilon(\Gamma^A)},\]
the generating function enumerating the edge subsets of $\Gamma$ by the Euler genus of their partial duals. It was observed in \cite{Gross_Mansour_Tucker_2020} that any connected ribbon graph has a partial dual that is a bouquet (one-vertex ribbon graph), so it suffices to focus on bouquets when studying the partial-dual Euler-genus polynomial, since $\pDe{\Gamma^A}(z)=\pDe{\Gamma}(z)$ for all $A\subseteq E(\Gamma)$ (this follows from the fact proved in \cite{Chmutov_2009} that $(\Gamma^A)^B=\Gamma^{A\triangle B}$ for all $B\subseteq E(\Gamma)$, where $\triangle$ is the symmetric difference). Yan and Jin \cite{Yan_Jin_2022} showed that the partial-dual Euler-genus polynomial of a bouquet is determined by its signed intersection graph.
\begin{prop}\label{prop:Yan_Jin_signed_intersection_graph}
    (Theorem 4.1 in \cite{Yan_Jin_2022}). If two bouquets $B_1$ and $B_2$ have isomorphic signed intersection graphs, then $\pDe{B_1}(z)=\pDe{B_2}(z).$
\end{prop}

In view of \autoref{prop:Yan_Jin_signed_intersection_graph}, Yan and Jin \cite{Yan_Jin_2022} defined the \textit{intersection polynomial} $\IP_G(z)$ of a signed intersection graph $G$ to be the partial-dual Euler-genus polynomial of any bouquet whose signed intersection graph is isomorphic to $G$. They obtained the following recurrence relation for the intersection polynomial that applies to positive leaves, allowing one to (in particular) recursively compute the intersection polynomial of any positive tree.

\begin{prop}\label{prop:Yan_Jin_recurrence}
    (Theorem 5.2 in \cite{Yan_Jin_2022}). Let $G$ be a signed intersection graph and let $v_1,v_2$ be adjacent vertices of $G$ such that $v_1$ is positive and has degree 1. Then
    \[\IP_{G}(z)=\IP_{G-v_1}(z)+2z^2\IP_{G-v_1-v_2}(z).\]
\end{prop}
For our purposes, it will be more natural to work equivalently with looped simple graphs instead of signed graphs; the two are in bijective correspondence by associating a negative sign at a vertex $v$ with the presence of a loop (self-edge) at $v$, and a positive sign with the absence of a loop.

Not every graph is the intersection graph of a bouquet (in fact, an excluded-minor characterization of intersection graphs was obtained by Bouchet in \cite{Bouchet_1994}); nevertheless, one can use the theory of delta-matroids to define a polynomial on all looped simple graphs that coincides with the intersection polynomial when restricted to intersection graphs. Namely, any looped simple graph $G$ has an associated binary delta-matroid $D(G)$, and Yan and Jin introduced the \textit{twist polynomial} $\T{D}(z)$ of a delta-matroid $D$ in \cite{Yan_Jin_2022_twist}. We say that the twist polynomial of $G$ is $\T{G}(z)=\T{D(G)}(z)$, the twist polynomial of its delta-matroid. That the twist polynomial extends the intersection polynomial was shown in \cite{Yan_Jin_2022_twist}.

\begin{definition}
Let $G_1$ and $G_2$ be looped simple graphs with respective vertices $v_1$ and $v_2$ that are either both unlooped or both looped. Define the \emph{one-point join} $(G_1,v_1)\vee (G_2,v_2)$ (which we will often shorten to $G_1\vee G_2$) to be the disjoint union of $G_1$ and $G_2$ with the vertices $v_1$ and $v_2$ identified.    
\end{definition}
When $G_1$ and $G_2$ are both unlooped (i.e., $G_1$ and $G_2$ are both simple graphs), $G_1\vee G_2$ is also known as the \textit{1-clique-sum} of $G_1$ and $G_2$. Any graph containing a cut vertex (a vertex whose removal increases the number of connected components) can be decomposed as a one-point join of two smaller graphs joined at the cut vertex.

Given a looped simple graph $G$ with a vertex $v$, the \textit{loop complementation} of $G$ at $v$ is the graph $G+v$ obtained from $G$ by adding (resp. removing) a loop at $v$ if $v$ was originally unlooped (resp. looped) in $G$. The main result of this paper is the recursion formula below which expresses the twist polynomial of $G_1\vee G_2$ in terms of the pairwise products of the twist polynomials of $G_1$, $G_1+v_1$, and $G_1-v_1$ and the twist polynomials of $G_2$, $G_2+v_2$, and $G_2-v_2$.

\begin{theorem}\label{thm:looped_recursion}
        If $(G_1,v_1)\vee (G_2,v_2)$ is a one-point join of two looped simple graphs, then
        \[
        (4z^2-2z-2)\T{G_1\vee G_2}=\begin{bmatrix}
            \T{G_1} & \T{G_1+v_1} & \T{G_1-v_1}
        \end{bmatrix}
        \begin{bmatrix}
            -(z+1) & z & 2z^2 \\
            z & -(z+1) & 2z^2 \\
            2z^2 & 2z^2 & -4z^2(z+1)
        \end{bmatrix}
        \begin{bmatrix}
            \T{G_2} \\ \T{G_2+v_2} \\ \T{G_2-v_2}
        \end{bmatrix}.
        \]
\end{theorem}
Moreover, we find that if $G_1-v_1$ or $G_2-v_2$ is unlooped, then \autoref{thm:looped_recursion} has a simpler form that involves only the twist polynomials of $G_1$, $G_1-v_1$, $G_2$, and $G_2-v_2$.

\begin{theorem}\label{thm:unlooped_recursion}
    If $(G_1,v_1)\vee (G_2,v_2)$ is a one-point join of two looped simple graphs such that either $G_1-v_1$ or $G_2-v_2$ is unlooped, then
    \[
    \T{G_1\vee G_2}=\frac{2z^2\T{G_1}\T{G_2-v_2}+2z^2\T{G_1-v_1}\T{G_2}-\T{G_1}\T{G_2}-4z^2\T{G_1-v_1}\T{G_2-v_2}}{2z^2-2}.
    \]
\end{theorem}
In particular, taking $G_1$ to be the graph consisting of a single unlooped vertex connected to $v_1$ in \autoref{thm:unlooped_recursion} gives a generalization of \autoref{prop:Yan_Jin_recurrence} to looped simple graphs.

Note that the contents of \autoref{thm:looped_recursion} and \autoref{thm:unlooped_recursion} can be applied to the special case of ribbon graphs by restricting to intersection graphs; in this setting, the one-point join can be interpreted as a suitable join on bouquets, and loop complementation corresponds to the partial Petrial. Translating instead to the language of (normal, binary) delta-matroids, the one-point join operation is described for set systems in Section \ref{sec:The Delta-Matroid Associated to a One-Point Join of Graphs}, and loop complementation was generalized to set systems in \cite{Brijder_Hoogeboom_2011}.

This paper is organized as follows. In Section \ref{sec:Background}, we review some definitions and prior results pertaining to ribbon graphs and delta-matroids. We show in Section \ref{sec:A Rank Formula for the Twist Polynomial on Looped Simple Graphs} that the twist polynomial on looped simple graphs can be computed using ranks of adjacency matrices. In Section \ref{sec:The Twist Polynomial of a One-Point Join of Graphs}, we prove \autoref{thm:looped_recursion} and \autoref{thm:unlooped_recursion}, along with two results on the twist polynomial of looped simple graphs evaluated at $-1/2$, and the twist polynomial of a graph with a single looped vertex. We also use \autoref{thm:unlooped_recursion} to compute the twist polynomial of the windmill graph. In Section \ref{sec:The Delta-Matroid Associated to a One-Point Join of Graphs}, we give a description of the delta-matroid $D(G_1\vee G_2)$ purely in terms of the delta-matroid structures of $D(G_1)$ and $D(G_2)$ by deriving a determinant identity for an ``almost block-diagonal'' matrix. Lastly, in Section \ref{sec:Leaf Recursion for Delta-Matroids}, we show that the leaf recursion formula in \autoref{prop:Yan_Jin_recurrence} holds for the twist polynomial on delta-matroids.

\section{Background}\label{sec:Background}
\subsection{Ribbon Graphs}
We briefly recall some basic notions from the theory of ribbon graphs; for a more detailed introduction, see \cite{Ellis-Monaghan_Moffatt_2013}. A \textit{ribbon graph} $\Gamma$ is a (possibly non-orientable) surface with boundary expressed as a union of vertex and edge discs such that
\begin{itemize}
    \item Vertices and edges intersect at disjoint line segments.
    \item Each such line segment lies on the boundary of exactly one vertex and edge.
    \item Each edge contains exactly two line segments.
\end{itemize}
We let $V(\Gamma)$ and $E(\Gamma)$ denote the sets of vertices and edges of $\Gamma$ respectively. For $A\subseteq E(\Gamma)$, Chmutov \cite{Chmutov_2009} defined the \textit{partial dual} $\Gamma^A$ of $\Gamma$ with respect to $A$ as the ribbon graph obtained from $\Gamma$ by sewing a disc into each boundary component of the spanning ribbon subgraph induced by $A$ (these new discs then become the vertices of $\Gamma^A$) and deleting the interiors of all vertices of $G$. In particular, $\Gamma^{E(\Gamma)}$ is the geometric dual (also called Euler-Poincar\'{e} dual) of $G$.

Given $A\subseteq E(\Gamma)$, the \textit{partial Petrial} is the ribbon graph $\Gamma^{\tau(A)}$ obtained from $\Gamma$ by adding a half-twist to each edge $e\in A$ (in discussions of the partial Petrial, the partial dual is also denoted as $\Gamma^{\delta(A)}$).

A \textit{spanning quasi-tree} of a connected ribbon graph $\Gamma$ is a spanning ribbon subgraph of $\Gamma$ with exactly one boundary component. If $\Gamma$ is disconnected, a \textit{spanning quasi-tree} of $\Gamma$ is a spanning ribbon subgraph $Q$ such that $Q$ has the same number of connected components as $\Gamma$, and the connected components of $Q$ are spanning quasi-trees of the connected components of $\Gamma$.

A \textit{bouquet} $B$ is a ribbon graph with exactly one vertex $v$. An edge $e\in E(B)$ is \textit{twisted} (or \textit{nonorientable}) if $e\cup v$ is homeomorphic to a M\"{o}bius band, and \textit{untwisted} (or \textit{orientable}) otherwise. A \textit{signed rotation system} \cite{Gross_Tucker_1987} of $B$ is obtained by choosing an orientation on the single vertex $v$ of $B$, which induces a cyclic ordering of the half-edges incident to $v$, and assigning a positive (resp. negative) sign to each untwisted (resp. twisted) edge of $B$. The \textit{signed intersection graph} of $B$ is the graph $I(B)$ whose vertices correspond to the edges of $B$, with two vertices of $I(B)$ being adjacent if and only if the half-edges of the two corresponding ribbon edges appear in alternating order in the signed rotation system of $B$; each vertex of $I(B)$ is given the corresponding sign in the signed rotation system. A signed intersection graph is called \textit{positive} if all of its vertices are assigned a positive sign. For the remainder of this paper, we will work exclusively with looped simple graphs in place of signed graphs.

\subsection{Delta-Matroids}
A \textit{set system} is a pair $S=(E,\mathcal{F})$ of a set $E$ called the \textit{ground set} and a collection $\mathcal{F}$ of subsets of $E$ called the \textit{feasible sets}. We will also use the notation $E(S)$ and $\mathcal{F}(S)$ to refer to the ground set and collection of feasible sets of $S$ respectively. We say that $S$ is a \textit{proper set system} if $\mathcal{F}$ is nonempty. A \textit{delta-matroid} is a proper set system $D=(E,\mathcal{F})$ that satisfies the Symmetric Exchange Axiom:
for all $X,Y\in\mathcal{F}$ and all $u\in X \triangle Y$, there exists $v\in X\triangle Y$ such that $X\triangle \{u,v\}\in\mathcal{F}$. Note that the possibility $u=v$ is allowed. 

A \textit{matroid} is a delta-matroid $M=(E,\mathcal{B})$ such that all elements of $\mathcal{B}$ are equicardinal. The \textit{rank} of a subset $A\subseteq E$ is
\[
r_M(A)=\max\{|A\cap B|:B\in\mathcal{B}\},
\]
and the \textit{nullity} of $A$ is $n_M(A)=|A|-r_M(A)$.

Given a delta-matroid $D=(E,\mathcal{F})$, let $\mathcal{F}_{\min}$ and $\mathcal{F}_{\max}$ denote the collections of feasible sets of $D$ with minimum and maximum cardinality respectively. Bouchet \cite{Bouchet_1989} defined the matroids $D_{\min}=(E,\mathcal{F}_{\min})$ and $D_{\max}=(E,\mathcal{F}_{\max})$, called the \textit{lower matroid} and \textit{upper matroid} of $D$ respectively. In \cite{Bouchet_1987_representability}, Bouchet defined the rank function on delta-matroids: for $A\subseteq E$, the rank of $A$ is
\[
\rho_D(A)=|E|-\min\{|A\triangle F|:F\in\mathcal{F}\}.
\]

An element $e\in E$ of a delta-matroid $D=(E,\mathcal{F})$ is called a \textit{coloop} if $e$ is contained in every feasible set of $D$, and $e$ is called a \textit{loop} if no feasible set of $D$ contains $e$. Bouchet and Duchamp \cite{Bouchet_1991_binary} defined the \textit{deletion} of $D$ by $e$ to be the delta-matroid
\[
D\setminus e=(E\setminus e,\{F:F\in\mathcal{F}\text{ and }F\subseteq E\setminus e\})
\]
if $e$ is not a coloop, and the \textit{contraction} of $D$ by $e$ to be 
\[
D/ e=(E\setminus e,\{F\setminus e:F\in\mathcal{F}\text{ and }e\in F\})
\]
if $e$ is not a loop. If $e$ is a coloop, then $D\setminus e$ is defined to be $D/e$, and if $e$ is a loop, then $D/e$ is defined to be $D\setminus e$. Any sequence of deletions and contractions applied to $D$ results in a delta-matroid independent of the order in which the deletions and contractions were done \cite{Bouchet_1991_binary}. For any $A\subseteq E$, the \textit{restriction} of $D$ to $A$ is $D|_A=D\setminus A^c$.

Given two delta-matroids $D_1=(E_1,\mathcal{F}_1)$ and $D_2=(E_2,\mathcal{F}_2)$ with $E_1\cap E_2=\emptyset$, the \textit{direct sum} of $D_1$ and $D_2$ is the delta-matroid 
\[
D_1\oplus D_2=(E_1\sqcup E_2,\{F_1\sqcup F_2:F_1\in\mathcal{F}_1,F_2\in\mathcal{F}_2\}).
\]
The \textit{width} of a delta-matroid $D=(E,\mathcal{F})$ is 
\[
w(D)=\max\{|F| : F\in\mathcal{F}\} - \min\{|F| : F\in\mathcal{F}\}.
\] 
For any $A\subseteq E$, Bouchet \cite{Bouchet_1987_symmetric_matroids} defined the \textit{twist} of $D$ by $A$ to be the delta-matroid
        \[D*A=(E,\{X\triangle A : X\in\mathcal{F}\}).\]
We say $D$ is a \textit{normal} delta-matroid if $\emptyset\in\mathcal{F}$. There is a convenient formula for the width of any twist of a normal delta-matroid. 
\begin{lemma}\label{lemma:width_of_twist_of_normal_delta_matroid}
    (Lemma 17 in \cite{Yan_Jin_2022_twist}). Let $D=(E,\mathcal{F})$ be a normal delta-matroid. Then for all $A\subseteq E$,
    \[
    w(D*A)=w(D|_A)+w(D|_{A^c}).
    \]
\end{lemma}

\begin{lemma}\label{lemma:wdith_of_restriction}
    (Proposition 5.39 in \cite{Chun_Moffatt_Noble_Rueckriemen_2019}). Let $D=(E,\mathcal{F})$ be a delta-matroid. Then for all $A\subseteq E$,
    \[
    w(D|_A)=\rho_D(A)-r_{D_{\min}}(A)-n_{D_{\min}}(E)+n_{D_{\min}}(A).
    \]
    In particular, if $D$ is normal, then
    \[
    w(D|_A)=|A|-\min\{|A\triangle F|:F\in\mathcal{F}\}.
    \]
\end{lemma}

\subsection{Ribbon-Graphic and Binary Delta-Matroids}

One can associate a delta-matroid to any ribbon graph, whereby many properties of ribbon graphs correspond to analogous properties of delta-matroids.

\begin{prop}\label{prop:ribbon_graphic_delta_matroid}
    (Chun, Moffatt, Noble, and Rueckriemen \cite{Chun_Moffatt_Noble_Rueckriemen_2019}). Let $\Gamma$ be a ribbon graph, and set
    \[
    \mathcal{F}=\{F\subseteq E(\Gamma):F\text{ is the edge set of a spanning quasi-tree of }\Gamma\}.
    \]
    Then the set system $D(\Gamma)=(E(\Gamma),\mathcal{F})$ is a delta-matroid, and we call $D(\Gamma)$ a ribbon-graphic delta-matroid. Moreover, $\varepsilon(\Gamma)=w(D(\Gamma))$ and for every $A\subseteq E(\Gamma)$, we have that $D(\Gamma^A)=D(\Gamma)*A$.
\end{prop}
In addition, if $B$ is a bouquet, then $D(B)$ is normal.

Let $\F$ be a field and consider a skew-symmetric matrix $M$ over $\F$ with rows and columns indexed by a finite set $E$. For $S\subseteq E$, let $M[S]$ denote the principal submatrix of $M$ obtained by deleting all rows and columns of $M$ indexed by elements of $E\setminus S$. Let
\[
\mathcal{F}=\{S\subseteq E:M[S]\text{ is nonsingular}\}.
\]
By convention, $M[\emptyset]$ is considered to be nonsingular. Bouchet \cite{Bouchet_1987_representability} showed that $D(M)=(E,\mathcal{F})$ is a delta-matroid. Any delta-matroid that has a twist isomorphic to $D(M)$ for some skew-symmetric matrix $M$ over $\F$ is called \textit{representable over} $\F$. 

\begin{lemma}\label{lemma:Bouchet_repr}
    (Bouchet \cite{Bouchet_1987_representability}). Let $\F$ be a field and suppose $D$ is a delta-matroid representable over $\F$. Then for all feasible sets $F$ of $D$, there exists a skew-symmetric matrix $M$ over $\F$ such that $D*F=D(M)$. 
\end{lemma}

We will mainly be interested in the case when $\F=\F_2$ is the field with two elements, in which case delta-matroids representable over $\F_2$ are called \textit{binary}. Over $\F_2$, the notions of symmetric and skew-symmetric matrices are identical, and such a matrix can be viewed as the adjacency matrix of a looped simple graph. A \textit{looped simple graph} $G$ is a graph that is allowed to have edges connecting a vertex to itself (called loops) but cannot have more than one edge between any given pair of vertices. We use $V(G)$ and $E(G)$ to denote the set of vertices and edges of $G$ respectively. The adjacency matrix $M$ of $G$ is defined as follows: for all distinct $u,v\in V(G)$, $M_{uv}=1$ if $\{u,v\}\in E(G)$ and $M_{uv}=0$ otherwise; $M_{vv}=1$ if $v$ has a loop and $M_{vv}=0$ otherwise. We define the delta-matroid associated to $G$ to be $D(G)=D(M)$, the binary delta-matroid obtained from the adjacency matrix of $G$.

Note that if $B$ is a bouquet then $D(B)=D(I(B))$.\footnote{Combined with \autoref{prop:ribbon_graphic_delta_matroid}, this gives an alternative proof of \autoref{prop:Yan_Jin_signed_intersection_graph}, which was proved in \cite{Yan_Jin_2022} using methods purely from topological graph theory.} This can be seen from the fact below.

\begin{lemma}\label{lemma:corank}
    (Mellor \cite{Mellor}). Let $B$ be a bouquet and let $M$ be the adjacency matrix of $I(B)$. Then the number of boundary components of $B$ is equal to $\corank M+1$.
\end{lemma}

One may recover a looped simple graph $G$ from its delta-matroid $D(G)=(V(G),\mathcal{F})$. Indeed, the adjacency matrix of $G$ is given by
\[
M_{vv}=\begin{cases}
    1 &\text{if } \{v\}\in\mathcal{F} \\
    0 &\text{otherwise}
\end{cases}
\]
for all $v\in V(G)$, and
\[
M_{uv}=\begin{cases}
    1 &\text{if }\{u,v\}\in \mathcal{F}\text{ and }\{u\},\{v\}\text{ are not both in }\mathcal{F},\text{ or if }\{u,v\}\not\in\mathcal{F}\text{ and }\{u\},\{v\}\in\mathcal{F} \\
    0 &\text{otherwise}
\end{cases}
\]
for all distinct $u,v\in V(G)$. \autoref{lemma:Bouchet_repr} implies that every normal, binary delta-matroid $D$ is of the form $D=D*\emptyset=D(M)$ for some symmetric matrix $M$ over $\F_2$. Thus, looped simple graphs are in one-to-one correspondence with normal, binary delta-matroids.

The loop complementation operation was generalized to set systems by Brijder and Hoogeboom in \cite{Brijder_Hoogeboom_2011}. Given a set system $S=(E,\mathcal{F})$ and $e\in E$, they defined the set system $S+e=(E,\mathcal{F}')$, where a subset $F\subseteq E$ belongs to $\mathcal{F}'$ if and only if
\[
\begin{cases}
    F\in\mathcal{F}\text{ or }F\setminus e\in\mathcal{F}\text{ but not both} &\text{if }e\in F \\
    F\in\mathcal{F} &\text{if }e\not\in F.
\end{cases}
\]

\begin{prop}\label{prop:loop_complementation_correspondence}
    (Theorem 8 in \cite{Brijder_Hoogeboom_2011}). Let $G$ be a looped simple graph. Then for all $v\in V(G)$,
    \[
    D(G+v)=D(G)+v.
    \]
\end{prop}

The partial Petrial corresponds to loop complementation: for a bouquet $B$ and $A\subseteq E(B)$, it is easy to see that $I(B^{\tau(A)})=I(B)+A$, the result of taking the loop complementation at each $v\in A$ on $I(B)$. Moreover, if $A$ is a spanning quasi-tree, then partial duality with respect to $A$ corresponds to an operation on the adjacency matrix of $I(B)$ called the principal pivot transform \cite{Tucker_1960, Brijder_Hoogeboom_2011}.

Yan and Jin \cite{Yan_Jin_2022_twist} defined the twist polynomial for delta-matroids, which extends the partial-dual Euler-genus polynomial. Namely, given a delta-matroid $D=(E,\mathcal{F})$, the \textit{twist polynomial} of $D$ is
\[
\T{D}(z)=\sum_{A\subseteq E}z^{w(D*A)}.
\]
It follows immediately from \autoref{prop:ribbon_graphic_delta_matroid} that for any ribbon graph $\Gamma$,
\[
\pDe{\Gamma}(z)=\T{D(\Gamma)}(z).
\]
Thus, $\T{D(G)}(z)=\IP_G(z)$ for all intersection graphs $G$, so the twist polynomial extends the intersection polynomial to a polynomial on all looped simple graphs. Henceforth, we use the notation $\T{G}(z)=\T{D(G)}(z)$ for looped simple graphs $G$.

Note that for all $A,B\subseteq E$, we have that $(D*A)*B)=D*(A\triangle B)$, so $\T{D*A}(z)=\T{D}(z)$. Since twisting $D$ by any of its feasible sets yields a normal delta-matroid, we may restrict to normal delta-matroids to study the twist polynomial, just as we restrict to bouquets to study the partial-dual Euler-genus polynomial.

The twist polynomial is multiplicative over the direct sum of two delta-matroids.
\begin{prop}\label{prop:twist_polynomial_of_direct_sum}
    (Proposition 7 in \cite{Yan_Jin_2022_twist}). If $D_1=(E_1,\mathcal{F}_1)$ and $D_2=(E_2,\mathcal{F}_2)$ are two delta-matroids with $E_1\cap E_2=\emptyset$, then
    \[
    \T{D_1\oplus D_2}(z)=\T{D_1}(z)\T{D_2}(z).
    \]
\end{prop}
In particular, since $D(G_1\sqcup G_2)=D(G_1)\oplus D(G_2)$ for any two looped simple graphs $G_1$ and $G_2$, the twist polynomial on graphs is multiplicative over disjoint union. We show in \autoref{thm:looped_recursion} that an analogous identity can be obtained in the more complicated setting of the one-point join.

\section[A Rank Formula for the Twist Polynomial on Looped Simple Graphs]{A Rank Formula for the Twist Polynomial on Looped \\ Simple Graphs}\label{sec:A Rank Formula for the Twist Polynomial on Looped Simple Graphs}
Here, we obtain an expression for the twist polynomial of a looped simple graph $G$ in terms of the ranks of principal submatrices of the adjacency matrix of $G$. Note that the expression for the twist polynomial that we derive in \autoref{prop:twist_polynomial_on_binary_delta_matroids} was obtained for the special case of intersection graphs by Cheng in \cite{Cheng_2025} using \autoref{lemma:corank}.

\begin{lemma}\label{lemma:rank_principal}
    Let $\F$ be a field and let $M$ be an $n\times n$ matrix over $\F$. If $M$ is symmetric or skew-symmetric, then
    \[
    \rank M=\max\{|S|:S\subseteq\{1,\dots,n\} \text{ and } M[S]\text{ is nonsingular}\}.
    \]
\end{lemma}
\begin{proof}
    The order of any nonsingular principal submatrix of $M$ cannot be greater than $\rank M$. If $M$ is symmetric or skew-symmetric, then $M$ has a nonsingular principal submatrix of order $\rank M$ (e.g., see Section 0.7.6 of \cite{Horn_Johnson_2012}).
\end{proof}

\begin{lemma}\label{lemma:width_is_rank}
    Let $\F$ be a field and suppose $M$ is a skew-symmetric matrix over $\F$. Then for all subsets $A$ of the ground set of $D(M)$,
    \[
    w(D(M)|_A)=\rank M[A].
    \]
\end{lemma}

\begin{proof}
    By \autoref{lemma:rank_principal} and the fact that $D(M[A])$ is normal, we have that
    \[
    w(D(M)|_A)=w(D(M[A]))=\max\{|S|:S\subseteq A\text{ and }M[S]\text{ is nonsingular}\}=\rank (M[A]).
    \]
\end{proof}

\begin{prop}\label{prop:twist_polynomial_on_binary_delta_matroids}
    Let $G$ be a looped simple graph with adjacency matrix $M$. Then for all $A\subseteq V(G)$,
    \[
    w(D(G)*A)=\rank(M[A])+\rank(M[A^c]).
    \]
    In particular,
    \[
    \T{G}(z)=\sum_{A\subseteq V(G)}z^{\rank(M[A])+\rank(M[A^c])}.
    \]
\end{prop}
\begin{proof}
    Note that $D(G)$ is a normal delta-matroid. Thus, for every $A\subseteq V(G)$, 
    \[
    w(D(G)*A)=w(D(G)|_A)+w(D(G)|_{A^c})=\rank(M[A])+\rank(M[A^c])
    \]
    by \autoref{lemma:width_of_twist_of_normal_delta_matroid} and \autoref{lemma:width_is_rank}.
\end{proof}

\section{The Twist Polynomial of a One-Point Join of Graphs}\label{sec:The Twist Polynomial of a One-Point Join of Graphs}

In this section, we prove the recursion formulas in \autoref{thm:looped_recursion} and \autoref{thm:unlooped_recursion}. The tools we develop yield two additional results. We get a relation for all looped simple graphs $G$ and all vertices $v$ of $G$ that relates the twist polynomials of $G$, $G+v$, and $G-v$ evaluated at $-1/2$. Furthermore, when $G$ is unlooped, we obtain an expression for the twist polynomial of $G+v$ in terms of the twist polynomials of $G$ and $G-v$. We conclude this section with an example showing how \autoref{thm:unlooped_recursion} can be used to compute the twist polynomial of the windmill graph.

The central tool needed to prove \autoref{thm:looped_recursion} lies in the next lemma, which allows us to compute, for a fixed $G_1$ and $v_1\in V(G_1)$, a recursion formula that expresses $\T{(G_1,v_1)\vee (G_2,v_2)}(z)$ as a $\Z[z]$-linear combination of $\T{G_2}(z)$, $\T{G_2+v_2}(z)$, and $\T{G_2-v_2}(z)$.

We introduce some notation. For a looped simple graph $G$, a vertex $v$ of $G$, and a value $\delta\in\{0,1\}$, we let $G+\delta v=G+v$ if $\delta=1$ and $G+\delta v=G$ if $\delta=0$. Given a one-point join $(G_1,v_1)\vee (G_2,v_2)$ of two looped simple graphs, let $V_1=V(G_1)$ and $V_2=V(G_2)$. We identify $V_1$ and $V_2$ with their images under the identification $v_1\sim v_2$, and refer to the identified vertex as $v$. We say that $G_1$ and $G_2$ are \textit{joined} at $v$. Let $M$ be the adjacency matrix of $G_1\vee G_2$, with rows and columns indexed by the vertices of $G$ in an order such that all vertices of $G_1- v$ precede $v$, and $v$ precedes all vertices of $G_2- v$. For $X\subseteq V_1$ with $v\in X$, write $M[X]=(x_{ij})_{i,j=1}^n$ (note that the index $n$ corresponds to the vertex $v$). Define $\delta_X=x_{nn}+1$ if $x_{nn}$ is in a pivot position of $M[X]$ and $\delta_X=x_{nn}$ otherwise. Let $P_1$ be the collection of $X\subseteq V_1$ such that $v\in X$ and 
        \[\begin{bmatrix}
        x_{1,n} \\
        \vdots \\
        x_{n-1,n}
    \end{bmatrix}\in\SPAN\left\{\begin{bmatrix}
        x_{1,1} \\
        \vdots \\
        x_{n-1,1}
    \end{bmatrix},\dots,\begin{bmatrix}
        x_{1,n-1} \\
        \vdots \\
        x_{n-1,n-1}
    \end{bmatrix}\right\}.
    \]
    We make the convention that $\{v\}\in P_1$. Let $P_2=\{X\subseteq V_1:v\in X\text{ and }X\not\in P_1\}$.
    
\begin{lemma}\label{lemma:recursion_existence}
    Let $G_1\vee G_2$ be a one-point join of two looped simple graphs, joined at a vertex $v$, and let $D=D(G_1\vee G_2)$. Then
    \[
    \T{G_1\vee G_2}(z)=\sum_{X\in P_1}z^{w(D|_{V_1-v}*(X-v))}\T{G_2+\delta_X v}(z)+2\sum_{X\in P_2}z^{w(D|_{V_1}*X)}\T{G_2-v}(z)
    \]
    Moreover, if $G_1-v$ is unlooped, then $\delta_X=0$ for all $v\in X\subseteq V(G_1)$. 
    
\end{lemma}

\begin{proof}
    For $X\subseteq V_1$ with $v\in X$, let 
    \[
    \tau_X=\{A\subseteq V:A\cap V_1=X\}
    \]
    and $M[X]=(x_{ij})_{i,j=1}^n$ with the index ordering discussed previously. For all $A\in \tau_X$,
    \[
    w(D*A)=\rank(M[A])+\rank(M[A^c])
    \]
    by \autoref{prop:twist_polynomial_on_binary_delta_matroids}. Since $v\not\in A^c$, we have that $M[A^c]$ is block-diagonal with two blocks, indexed by $A^c\cap V_1=X^c$ and $A^c\cap V_2$ respectively. Thus
    \[
    \rank(M[A^c])=\rank(M[X^c])+\rank(M[A^c \cap V_2]).
    \]
    We may write
    \[M[A]=
\begin{bmatrix}
    x_{11} & \cdots & x_{1n} &    \\
    \vdots & \ddots & \vdots &    \\
    x_{n1} & \cdots & x_{nn} & y^T  \\
     &  & y & Z
\end{bmatrix},\]
where the empty entries indicate zeros, $Z=M[A\cap (V_2-v)]$, and $M[A\cap V_2]=\begin{bmatrix}
    x_{nn} & y^T \\
    y & Z
\end{bmatrix}$.

    \emph{Case 1}. Suppose $X\in P_1$. Then
    $$\begin{bmatrix}
        x_{1,n} \\
        \vdots \\
        x_{n-1,n}
    \end{bmatrix}=\begin{bmatrix}
        x_{1,j_1} \\
        \vdots \\
        x_{n-1,j_1}
    \end{bmatrix}+\cdots+\begin{bmatrix}
        x_{1,j_r} \\
        \vdots \\
        x_{n-1,j_r}
    \end{bmatrix}$$
    for some $1\leq j_1<j_2<\cdots<j_r\leq n-1$. So, through a series of elementary row and column operations, we can transform $M[A]$ into the following block-diagonal matrix:
    \[M[A]=
\begin{bmatrix}
    x_{11} & \cdots & x_{1,n-1} &  &   \\
    \vdots & \ddots & \vdots &  &    \\
    x_{n-1,1} & \cdots & x_{n-1,n-1} &   \\
     &  &  & x_{nn}+\delta & y^T \\
     & & & y & Z
\end{bmatrix}\]
where $\delta=x_{j_1j_1}+\cdots+x_{j_rj_r}$. Note that $\delta=\delta_X$, so $\delta$ is independent of the choice of linear combination that witnesses $X$ belonging to $P_1$. Evidently, if $G_1-v$ is unlooped, then $\delta_X=0$. Let $M'$ be the adjacency matrix of $G+\delta_X v$, and let $D'$ be the binary delta-matroid obtained from $M'$. We have that
\[
\rank(M[A])=\rank(M[X-v])+\rank(M' [A\cap V_2]),
\]
so
\[
w(D*A)=w(D|_{V_1- v}*(X-v))+w(D'|_{V_2}*(A\cap V_2)).
\]
Thus,
\begin{align*}
    \sum_{A\in\tau_X}z^{w(D*A)}
    &=z^{w(D|_{V_1-v}*(X-v))}\sum_{A\in\tau_X}z^{w(D'|_{V_2}*(A\cap V_2))} \\
    &=z^{w(D|_{V_1-v}*(X-v))}\sum_{v\in B\subseteq V_2}z^{w(D'|_{V_2}*B)} \\
    &=\frac{1}{2}z^{w(D|_{V_1-v}*(X-v))}\T{G_2+\delta_X v}(z).
\end{align*}

    \emph{Case 2}. Suppose $X\in P_2$. Then the system of linear equations over $\F_2$ corresponding to the augmented matrix
    \[\begin{amatrix}{3}
   x_{1,1} & \dots & x_{1,n-1} & x_{1,n} \\
   \vdots & \ddots & \vdots & \vdots \\
   x_{n-1,1} & \dots & x_{n-1,n-1} & x_{n-1,n}
 \end{amatrix}\]
 is inconsistent, so the rightmost column of this augmented matrix is a pivot column (contains a row-leading 1 in the reduced row echelon form), which means
 \[[0,0\dots,0,1]\in\SPAN\left\{[
        x_{1,1},\dots, x_{1,n}]
    ,\dots,[
        x_{n-1,1},\dots, x_{n-1,n}]
    \right\}.\]
    We can then transform $M[A]$ into the following block-diagonal matrix using elementary row and column operations:
    \[M[A]=
\begin{bmatrix}
    x_{11} & \cdots & x_{1n} &    \\
    \vdots & \ddots & \vdots &    \\
    x_{n1} & \cdots & x_{nn} &   \\
     &  &  & Z
\end{bmatrix}.\]
Thus, $\rank(M[A])=\rank(M[X])+\rank(M[A\cap(V_2-v)])$, so
\[
w(D*A)=w(D|_{V_1}*X)+w(D|_{V_2-v}*(A\cap(V_2-v))).
\]
Then
\[
\sum_{A\in\tau_X}z^{w(D*A)}=z^{w(D|_{V_1}*X)}\sum_{A\in\tau_X}z^{w(D|_{V_2-v}*(A\cap(V_2-v)))}=z^{w(D|_{V_1}*X)}\T{G_2-v}(z).
\]
Altogether, we get that
\begin{align*}
    \T{G_1\vee G_2}(z)
    &=2\sum_{v\in X\subseteq V_1}\sum_{A\in\tau_X}z^{w(D*A)} \\
&=\sum_{X\in P_1}z^{w(D|_{V_1-v}*(X-v))}\T{G_2+\delta_X v}(z)+2\sum_{X\in P_2}z^{w(D|_{V_1}*X)}\T{G_2-v}(z)
\end{align*}
with $\delta_X=0$ for all $v\in X\subseteq V_1$ if $G_1-v$ is unlooped.
\end{proof}

Next, we apply \autoref{lemma:recursion_existence} to two specific cases of $G_1$, which will be used to derive the coefficients in \autoref{thm:looped_recursion}. Note that part (a) of \autoref{prop:leaf_recursion} extends Yan and Jin's recursion for signed intersection graphs (\autoref{prop:Yan_Jin_recurrence}) to looped simple graphs; Cheng also obtained an extension of \autoref{prop:Yan_Jin_recurrence} to (unlooped) simple graphs in Theorem 4.1 of \cite{Cheng_2025}.
\begin{prop}\label{prop:leaf_recursion}
    Let $G_1$ be a graph consisting of two vertices $w$ and $v_1$ connected by an edge. Consider $(G_1,v_1)\vee (G_2,v_2)$ for any looped simple graph $G_2$ with a vertex $v_2$ having the same loop status as $v_1$.
    \begin{enumerate}[(a)]
        \item If $w$ is unlooped, then
        \[
        \T{G_1\vee G_2}=\T{G_2}+2z^2\T{G_2-v_2}.
        \]
        \item If $w$ is looped, then
        \[
        \T{G_1\vee G_2}=z\T{G_2}+z\T{G_2+v_2}.
        \]
    \end{enumerate}
\end{prop}

\begin{proof}
    First suppose $w$ is unlooped. Let $M$ be the adjacency matrix of $G$. Then $M[V_1]=\begin{bmatrix}
        0&1 \\
        1&\lambda
    \end{bmatrix}$ where $\lambda=1$ if $v_1$ has a loop and $\lambda=0$ otherwise. So, $P_1=\{\{v_1\}\}$ and $P_2=\{\{w,v_1\}\}$. Now, 
    \[
    w(D|_{V_1-v}*(\{v\}-v))=w(D|_{V_1-v})=\rank(M[\{w\}])=0
    \]
    and
    \[
    w(D|_{V_1}*\{w,v\})=w(D|_{V_1})+w(D|_\emptyset)=\rank(M[V_1])=2.
    \]
    Note that $\delta_{\{v\}}=0$ because $v$ corresponds to a pivot position of $M[\{v\}]$ if and only if $\lambda=1$. \autoref{lemma:recursion_existence} then gives that
    $$\T{G_1\vee G_2}=\T{G_2}+2z^2\T{G_2-v}.$$
    Now suppose $w$ is looped, so that $M[V_1]=\begin{bmatrix}
        1&1 \\
        1&\lambda
    \end{bmatrix}$. Then $P_1=\{\{v\},\{w,v\}\}$ and $P_2=\emptyset$. We find that
    \[
    w(D|_{V_1-v}*(\{v\}-v))=w(D|_{V_1-v})=\rank(M[\{w\}])=1,
    \]
    \[
    w(D|_{V_1-v}*(\{w,v\}-v))=w(D|_{V_1-v})+w(D|_\emptyset)=1,
    \]
    $\delta_{\{v\}}=0$ as before, and $\delta_{\{w,v\}}=1$ since $v$ corresponds to a pivot position of $M[\{w,v\}]$ if and only if $\lambda=0$. Thus,
    \[
        \T{G_1\vee G_2}=z\T{G_2}+z\T{G_2+v}
        \]
        by \autoref{lemma:recursion_existence}.
\end{proof}

\begin{lemma}\label{lemma:recursion_uniqueness}
    Let $G_1$ be a looped simple graph and $v_1\in V(G_1)$. Suppose there exist $Q_1,Q_2,Q_3\in\Z[z]$ such that for all looped simple graphs $G_2$ and $v_2\in V(G_2)$ with the same loop status as $v_1$, the twist polynomial of $(G_1,v_1)\vee (G_2,v_2)$ can be expressed as
    \[
        \T{G_1\vee G_2}={Q_1}\T{G_2}+{Q_2}\T{G_2+v_2}+{Q_3}\T{G_2-v_2}.\tag{$*$}
    \]
    Then
    \[{Q_1}=\frac{2z^2\T{G_1-v_1}-(z+1)\T{G_1}+z\T{G_1+v_1}}{2(2z^2-z-1)}\]
\[{Q_2}=\frac{2z^2\T{G_1-v_1}-(z+1)\T{G_1+v_1}+z\T{G_1}}{2(2z^2-z-1)}\]
\[{Q_3}=\frac{z^2(\T{G_1}+\T{G_1+v_1})-2(z^3+z^2)\T{G_1-v_1}}{2z^2-z-1}.\]
\end{lemma}

\begin{proof}
Suppose first that $v_1$ is unlooped. Setting $G_2$ to be the graphs consisting of just $v_2$, and of an unlooped/looped vertex connected to $v_2$, and using \autoref{prop:leaf_recursion} (with $G_1$ and $G_2$ swapped) together with $(*)$ gives a system of 3 equations that completely determines $Q_1,Q_2,Q_3$.
        \[
            \begin{bmatrix}
            2 & 2z & 1 \\
            2+2z^2 & 2z+2z^2 & 2 \\
            2z+2z^2 & 2z+2z^2 & 2z
        \end{bmatrix}\begin{bmatrix}
            Q_1 \\
            Q_2 \\
            Q_3
        \end{bmatrix}=\begin{bmatrix}
            \T{G_1} \\
            \T{G_1}+2z^2\T{G_1-v_1} \\
            z\T{G_1}+z\T{G_1+v_1}
        \end{bmatrix}\tag{$**$}
        \]
        Solving gives the desired expressions for $Q_1,Q_2,Q_3$.
        
    If $v_1$ is looped, then we get the system of equations
    \[
    \begin{bmatrix}
        2z & 2 & 1 \\
        2z+2z^2 & z+2z^2 & 2 \\
        2z+2z^2 & 2z+2z^2 & 2z
    \end{bmatrix}\begin{bmatrix}
        Q_1 \\
        Q_2 \\
        Q_3
    \end{bmatrix}=\begin{bmatrix}
        \T{G_1} \\
        \T{G_1}+2z^2\T{G_1-v_1} \\
        z\T{G_1}+z\T{G_1+v_1}
    \end{bmatrix}.
    \]
    Up to a permutation swapping $Q_1$ and $Q_2$, this is identical to $(**)$. Note however that swapping $G_1$ and $G_1+v_1$ in the expressions obtained for $Q_1,Q_2,Q_3$ in the case of unlooped $v_1$ has the effect of swapping $Q_1$ and $Q_2$ while leaving $Q_3$ unchanged. Therefore, the recursion formula for the unlooped case holds for the looped case as well.
\end{proof}

The proof of \autoref{thm:looped_recursion} is now immediate.
\begin{proof}
    By \autoref{lemma:recursion_existence}, there exist ${Q_1},{Q_2},{Q_3}\in\mathbb{Z}[z]$ depending on $G_1$ and $v_1$ such that
    \[
    \T{G_1\vee G_2}={Q_1}\T{G_2}+{Q_2}\T{G_2+v_2}+{Q_3}\T{G_2-v_2}.
    \]
    \autoref{lemma:recursion_uniqueness} then gives the explicit forms for $Q_1,Q_2,Q_3$, which yields the desired result upon simplifying.

\end{proof}

\begin{cor}\label{cor:-1/2}
    For any looped simple graph $G$ and any vertex $v\in V(G)$,
    \[\T{G}\left(-\frac{1}{2}\right)+\T{G+v}\left(-\frac{1}{2}\right)-\T{G-v}\left(-\frac{1}{2}\right)=0.\]
\end{cor}

\begin{proof}
    Replacing $G_1$ by $G$ in \autoref{lemma:recursion_existence} and \autoref{lemma:recursion_uniqueness}, we have that $Q_1',Q_2',Q_3'$  are all polynomials, so their numerators must each have $-1/2$ as a root since $2z^2-z-1=(z-1)(2z+1)$.
\end{proof}
As a remark, note that divisibility by $z-1$ in the preceding proof is a consequence of the fact that $\T{G}(1)=2^{|V(G)|}$ for all looped simple graphs $G$. It would be interesting to find an interpretation of the twist polynomial evaluated at $-1/2$ so that the result in \autoref{cor:-1/2} can be better understood.

\begin{cor}\label{cor:loop_complementation_of_unlooped_graph}
    If $G$ is an unlooped simple graph and $v$ is any vertex of $G$, then
    \[
    \T{G+v}=\frac{z\T{G}+2z^2\T{G-v}}{z+1}
    \]
\end{cor}

\begin{proof}
    By taking the one-point join of $G$ with any graph at $v$, we see from \autoref{lemma:recursion_existence} and \autoref{lemma:recursion_uniqueness} that $Q_2=0$ and 
    \[
    Q_2=\frac{2z^2\T{G-v}-(z+1)\T{G+v}+z\T{G}}{2(2z^2-z-1)}.
    \]
\end{proof}

We now prove \autoref{thm:unlooped_recursion}.
\begin{proof}
    Since $G_1$ and $G_2$ can be interchanged in the formula to be proved, we may assume $G_1-v_1$ is unlooped without loss of generality. By \autoref{lemma:recursion_existence}, \autoref{lemma:recursion_uniqueness}, and \autoref{cor:loop_complementation_of_unlooped_graph}, we get that $Q_2=0$,
    \[
    Q_1=\frac{2z^2\T{G_1-v_1}-(z+1)\T{G_1}+z\T{G_1+v_1}}{2(2z^2-z-1)}=\frac{2z^2\T{G_1-v_1}-\T{G_1}}{2z^2-2},
    \]
    and
    \[
    Q_3=\frac{z^2(\T{G_1}+\T{G_1+v_1})-2(z^3+z^2)\T{G_1-v_1}}{2z^2-z-1}=\frac{2z^2\T{G_1}-4z^2\T{G_1-v_1}}{2z^2-2}.
    \]
    Substituting into $\T{G_1\vee G_2}=Q_1\T{G_2}+Q_3\T{G_2-v_2}$ gives the result.    
\end{proof}

Next, we use \autoref{thm:unlooped_recursion} to compute the twist polynomial of the windmill graph. For $n\geq 2$ and $m\geq 1$, the windmill graph $K_n^{(m)}=K_n\vee \cdots\vee K_n$ is formed by taking the disjoint union of $m$ copies of the complete graph $K_n$ on $n$ vertices and identifying one vertex from each copy into one common vertex. The twist polynomial for the case $m=1$ was studied by Yan and Jin in \cite{Yan_Jin_2022, Yan_Jin_2021} as a counterexample to a conjecture made by Gross, Mansour, and Tucker in \cite{Gross_Mansour_Tucker_2020}.
\begin{prop}\label{prop:complete_graph}
    (Theorem 23 in \cite{Yan_Jin_2021}). For all integers $n\geq 1$,
    \[
    \T{K_n}(z)=\begin{cases}
        2^{n-1}z^n+2^{n-1}z^{n-2} &\text{if }n\text{ is even} \\
        2^nz^{n-1} &\text{if }n\text{ is odd.}
    \end{cases}
    \]
\end{prop}

\begin{prop}{\label{prop:example}}
Let $n\geq 2$ and $m\geq 1$ be integers. Then
\[
    \T{K_n^{(m)}}(z)=\begin{cases}
        2^{m(n-2)+1}z^{m(n-2)}[(2^m-1)z^2+1] &\text{if }n\text{ is even} \\
        2^{m(n-2)+1}z^{m(n-3)+2}[(z^2+1)^m-z^{2m}+z^{2m-2}] &\text{if }n\text{ is odd.}
    \end{cases}
\]
\end{prop}

\begin{proof}
For ease of notation, let $W_{a,b}=\T{K_a^{(b)}}(z)$ for $a\geq2$, $b\geq1$. Applying \autoref{thm:unlooped_recursion} to $K_n^{(m)}=K_n^{(m-1)}\vee K_n$ for $m\geq 2$ and using \autoref{prop:twist_polynomial_of_direct_sum} gives
\begin{align*}
    (2z^2-2)W_{n,m}&=2z^2 W_{n,m-1}W_{n-1,1}+2z^2 W_{n-1,1}^{m-1} W_{n,1}-W_{n,m-1}W_{n,1}-4z^2W_{n-1,1}^m \\
    &=(2z^2 W_{n-1,1}-W_{n,1})W_{n,m-1}+2z^2(W_{n,1}-2W_{n-1,1})W_{n-1,1}^{m-1}.
\end{align*}

\emph{Case 1}. Suppose $n$ is even. Then by \autoref{prop:complete_graph},
\begin{align*}
    (2z^2-2)W_{n,m}&=2^{n-1}z^{n-2}(z^2-1)W_{n,m-1}+2^{n}z^n(z^2-1)W_{n-1,1}^{m-1} \\
    W_{n,m}&=2^{n-2}z^{n-2}(W_{n,m-1}+2z^2(2^{n-1}z^{n-2})^{m-1}) \\
    W_{n,m}&=2^{n-2}z^{n-2}W_{n,m-1}+2^{m(n-1)}z^{m(n-2)+2}. \tag{$*$}
\end{align*}
Solving the recurrence relation given by $(*)$ yields 
\[
W_{n,m}=2^{m(n-2)+1}z^{m(n-2)}[(2^m-1)z^2+1].
\]

\emph{Case 2}. Suppose now that $n$ is odd. Then
\begin{align*}
    (2z^2-2)W_{n,m}&=(2^{n-1}z^{n+1}-2^{n-1}z^{n-1})W_{n,m-1}+(2^{n}z^{n+1}-2^nz^{n-1})W_{n-1,1}^{m-1} \\
    W_{n,m}&=2^{n-2}z^{n-1}W_{n,m-1}+2^{n-1}z^{n-1}(2^{n-2}z^{n-1}+2^{n-2}z^{n-3})^{m-1}. \tag{$**$}    
\end{align*}
Solving $(**)$ then gives
\[
W_{n,m}=2^{m(n-2)+1}z^{m(n-3)+2}[(z^2+1)^m-z^{2m}+z^{2m-2}].
\]
\end{proof}

Observe that in the $m=2$ case, the two formulas in \autoref{prop:example} for even and odd $n$ are identical.

\begin{cor}
    For all integers $n\geq2$,
    \[
    \T{K_n\vee K_n}(z)=2^{2n-3}z^{2n-4}(3z^2+1).
    \]
\end{cor}

\section{The Delta-Matroid Associated to a One-Point Join of Graphs}\label{sec:The Delta-Matroid Associated to a One-Point Join of Graphs}
In this section, we give a description of the binary delta-matroid associated to the one-point join of two looped simple graphs purely in terms of feasible sets. We obtain this description from \autoref{lemma:det_identity}, which computes the determinant of a block matrix with two blocks that intersect at one entry. \autoref{lemma:det_identity} specializes to Corollary 2b in \cite{Schwenk} when $\det M$ is the characteristic polynomial of an unlooped graph. We present two different proofs of \autoref{lemma:det_identity}: the first is similar in nature to the proof in \cite{Schwenk} and uses a graph-theoretical formulation of the determinant observed by Harary, and the second uses the Schur determinant identity and the Zariski topology of affine space.

\begin{lemma}\label{lemma:linear_subgraphs}
    (Harary \cite{Harary}). Let $M$ be a square matrix over a field $\F$, viewed as the adjacency matrix of a looped digraph $D$ with edge weights belonging to $\F$. Denote by $\mathcal{L}(D)$ the family of \textit{linear subgraphs} of $D$, i.e. subgraphs consisting of directed cycles such that every vertex of $D$ is contained in exactly one such cycle (equivalently, subgraphs in which every vertex of $D$ has indegree and outdegree 1). For $G\in\mathcal{L}(D)$, let $n(G)$ be the number of even-length cycles in $G$, and let $p_G$ be the product of the edge weights present in $G$. Then
    \[
    \det M=\sum_{G\in\mathcal{L}(D)}(-1)^{n(G)}p_G.
    \]
\end{lemma}

\begin{lemma}\label{lemma:det_identity}
    Let $\F$ be a field and let $n,m\in\N$. Suppose $M$ is a matrix over $\F$ of the form
    \[
    M=\begin{bmatrix}
        A' & u & 0 \\
        v^T & c & x^T \\
        0 & y & B'
    \end{bmatrix},
    \]
    where $A'$ is an $n\times n$ matrix, $B'$ is an $m\times m$ matrix, $u$ and $v$ are $n$-dimensional column vectors, $x$ and $y$ are $m$-dimensional column vectors, and $c\in\F$. Let
    \[
    A=\begin{bmatrix}
        A' & u \\
        v^T & c
    \end{bmatrix}\;\;\;\;\text{and}\;\;\;\;B=\begin{bmatrix}
        c & x^T \\
        y & B'
    \end{bmatrix}.
    \]
    Then
    \[
    \det M = \det A\det B'+\det A'\det B-c\det A'\det B'.
    \]
\end{lemma}

\begin{proof}
    By abuse of notation, we will use the same symbol to refer to an edge-weighted, looped digraph and to its adjacency matrix. The vertex $t$ of $M$ with loop weight $c$ is a cut vertex, so any cycle containing $t$ is contained in either $A$ or $B$. Let $\tau_A$ and $\tau_B$ be the collections of $G\in\mathcal{L}(M)$ such that the cycle in $G$ containing $t$ is contained in $A$ and $B$, respectively. Let $\tau_t$ be the collection of $G\in\mathcal{L}(M)$ such that the cycle in $G$ containing $t$ consists of just the loop at $t$. Then $\mathcal{L}(M)=\tau_A\cup\tau_B$ and $\tau_A\cap\tau _B=\tau_t$. Since $t$ is a cut vertex, we obtain the decompositions
    \begin{align*}
        \sum_{G\in\tau_t}(-1)^{n(G)}p_G&=c(-1)^{n(G-V(B))}p_{G-V(B)}(-1)^{n(G-V(A))}p_{G-V(A)} \\
        &=c\left(\sum_{G'\in\mathcal{L}(A')}(-1)^{n(G')}p_{G'}\right)\left(\sum_{G''\in\mathcal{L}(B')}(-1)^{n(G'')}p_{G''}\right) \\
        &=c\det A'\det B',
    \end{align*}
    and
    \begin{align*}
        \sum_{G\in\tau_A}(-1)^{n(G)}p_G&=(-1)^{n(G-V(B'))}p_{G-V(B')}(-1)^{n(G-V(A))}p_{G-V(A)} \\
        &=\left(\sum_{G'\in\mathcal{L}(A)}(-1)^{n(G')}p_{G'}\right)\left(\sum_{G''\in\mathcal{L}(B')}(-1)^{n(G'')}p_{G''}\right) \\
        &=\det A\det B'.
    \end{align*}
    Similarly, $\sum_{G\in\tau_B}(-1)^{n(G)}p_G=\det A'\det B$. By \autoref{lemma:linear_subgraphs},
    \begin{align*}
        \det M&=\sum_{G\in\mathcal{L}(M)}(-1)^{n(G)}p_G \\
        &=\sum_{G\in\tau_A}(-1)^{n(G)}p_G+\sum_{G\in\tau_B}(-1)^{n(G)}p_G-\sum_{G\in\tau_t}(-1)^{n(G)}p_G \\
        &=\det A\det B'+\det A'\det B-c\det A'\det B'.
    \end{align*}    
\end{proof}
The previous proof uses the one-point join condition to simplify the linear subgraphs of $M$. In the following alternative proof of \autoref{lemma:det_identity}, the one-point join condition emerges in the additivity of the determinant for $1\times1$ matrices.

\begin{proof}
    By replacing $\F$ with its algebraic closure, we may assume $\F$ is algebraically closed. We first show that the formula holds whenever $A'$ and $B'$ are invertible. Applying the Schur determinant identity repeatedly gives
    \[
    \det A=(\det A')(c-v^T(A')^{-1}u),
    \]
    \[
    \det B=(\det B')(c-x^T(B')^{-1}y),
    \]
    and
    \begin{align*}
        \det M&=\det A'\det\left(B-\begin{bmatrix}
        v^T \\
        0
    \end{bmatrix}(A')^{-1}\begin{bmatrix}
        u & 0
    \end{bmatrix}\right) \\
    &=\det A'\det\begin{bmatrix}
        c-v^T(A')^{-1}u & x^T \\
        y & B'
    \end{bmatrix} \\
    &=(\det A'\det B')(c-v^T(A')^{-1}u-x^T (B')^{-1}y) \\
    &=\det A\det B'+\det A'\det B-c\det A'\det B'.
    \end{align*}
    For every invertible $A'$, the collection $Z$ of $m\times m$ matrices $B'$ satisfying the above expression for $\det M$ forms a Zariski closed subset of the affine space $\mathbb{A}^{m^2}$ over $\F$, since the determinant is a polynomial in the matrix entries. Since $Z$ contains the Zariski open set of all invertible $m\times m$ matrices, it follows that $Z=\mathbb{A}^{m^2}$ since $\mathbb{A}^{m^2}$ is irreducible. Applying a similar argument for an arbitrary (possibly non-invertible) $B'$ and varying $A'$ gives the result.
\end{proof}

For any proposition $P$, let $(P)\in\F_2$ be the value 1 if $P$ is true, and 0 if $P$ is false. So, for any two propositions $P$ and $Q$, we have that $(P)+(Q)=(P\text{ XOR }Q)$ and $(P)(Q)=(P\text{ AND }Q)$.
\begin{definition}\label{def:set_system_one-point_join}
    Let $S_1=(E_1,\mathcal{F}_1)$ and $S_2=(E_2,\mathcal{F}_2)$ be two set systems with $E_1\cap E_2=\emptyset$ and with distinguished elements $e_1\in E_1$ and $e_2\in E_2$ such that either both $\{e_1\}\in\mathcal{F}_1$ and $\{e_2\}\in\mathcal{F}_2$, or both $\{e_1\}\not\in\mathcal{F}_1$ and $\{e_2\}\not\in\mathcal{F}_2$. Let $E$ be the quotient set $(E_1\sqcup E_2)/\sim$ under the identification $e_1\sim e_2$. We identify $E$ with the set $[(E_1\setminus e_1)\cup(E_2\setminus e_2)]\sqcup\{e\}$. For $F\subseteq E$ and $i=1,2$, let $F_i$ be the preimage of $F$ in $E_i$ under the identification, and let $F_i'$ be the preimage of $F$ in $E_i\setminus e_i$. Explicitly, $F_i'=F\cap E_i$, and
    \[
    F_i=\begin{cases}
        F_i' &\text{if }e\not\in F \\
        F_i'\cup\{e_i\} &\text{if }e\in F.
    \end{cases}
    \]
    Let $\mathcal{F}$ be the collection of $F\subseteq E$ satisfying
    \[
    \begin{cases}
        1=(F_1\in \mathcal{F}_1)(F_2\in\mathcal{F}_2) &\text{if }e\not\in F \\
        1=(F_1\in\mathcal{F}_1)(F_2'\in\mathcal{F}_2)+(F_1'\in\mathcal{F}_1)(F_2\in\mathcal{F}_2)+(\{e_1\}\in\mathcal{F}_1)(F_1'\in\mathcal{F}_1)(F_2'\in\mathcal{F}_2) &\text{if }e\in F.
    \end{cases}
    \]
    Define $(S_1,e_1)\vee (S_2,e_2)$ (or $S_1\vee S_2$ for short) to be the set system $(E,\mathcal{F})$.
\end{definition}

\begin{prop}
Suppose $(G_1,v_1)\vee (G_2,v_2)$ is a one-point join of two looped simple graphs joined at $v\in V(G_1\vee G_2)$. Then
\[
D((G_1,v_1)\vee (G_2,v_2))=(D(G_1),v_1)\vee (D(G_2),v_2).
\]
In particular, if $D_1$ and $D_2$ are normal, binary delta-matroids, then $D_1\vee D_2$ is a normal, binary delta-matroid.   
\end{prop}

\begin{proof}
Let $M$ be the adjacency matrix of $G_1\vee G_2$. A subset $F\subseteq V(G_1\vee G_2)$ is feasible in $D(G_1\vee G_2)$ if and only if $\det M[F]=1$. By \autoref{lemma:det_identity}, This is equivalent to
\[
\begin{cases}
    1=\det M[F_1]\det M[F_2] &\text{if }v\not\in F \\
    1=\det M[F_1]\det M[F_2']+\det M[F_1']\det M[F_2]+\det M[\{v\}]\det M[F_1']\det M[F_2'] &\text{if }v\in F,
\end{cases}
\]
which is evidently equivalent to $F$ being a feasible set of $D(G_1)\vee D(G_2)$. The last statement of the proposition follows from the fact that \autoref{lemma:Bouchet_repr} implies that every normal, binary delta-matroid is of the from $D(G)$ for some looped simple graph $G$. 
\end{proof}

\section{Leaf Recursion for Delta-Matroids}\label{sec:Leaf Recursion for Delta-Matroids}
We now prove an extension of \autoref{prop:Yan_Jin_recurrence} to the twist polynomial on delta-matroids. The next lemma gives us control on the type of feasible set that best approximates (in terms of symmetric difference distance) an arbitrary subset of the ground set of a delta-matroid.
\begin{lemma}\label{lemma:symmetric_difference_distance}
Let $D=(E,\mathcal{F})$ be a delta-matroid and let $X\subseteq E$. If $Y\subseteq X$ and $Z\subseteq X^c$ and there exists $G\in\mathcal{F}$ such that $Y\subseteq G$ and $Z\subseteq G^c$, then there exists $F\in\mathcal{F}$ such that $Y\subseteq F$ and $Z\subseteq F^c$ and $|X\triangle F|=\min\{|X\triangle F'|:F'\in\mathcal{F}\}$.
\end{lemma}
\begin{proof}
    We proceed by induction on $|Y\cup Z|$. The base case $|Y\cup Z|=0$ is trivial, so suppose $|Y\cup Z|>0$ and let $e\in Y\cup Z$. Suppose $e\in Y$ (the case $e\in Z$ is similar). Then $Y\setminus e\subseteq Y\subseteq X,G$ and $Z\subseteq X^c, G^c$ so by induction, there exists $H\in\mathcal{F}$ with $Y\setminus e\subseteq H$ and $Z\subseteq H^c$ and $|X\triangle H|=\min\{|X\triangle F'|:F'\in\mathcal{F}\}$. If $e\in H$ then we're done by setting $F\coloneq H$. Otherwise, $e\in G\triangle H$ so there is $u\in G\triangle H$ such that $F\coloneq H\triangle\{e,u\}\in\mathcal{F}$. If $u=e$ then minimality of $|X\triangle H|$ is contradicted. So, $u\neq e$ and
\[
|X\triangle F| = |X\triangle H\triangle \{e\}\triangle \{u\}| = |X\triangle H\triangle \{e\}|\pm1 = |X\triangle H|-1\pm1 \leq|X\triangle H|
\]
    so $|X\triangle F|=\min\{|X\triangle F'|:F'\in\mathcal{F}\}$. Moreover, $Y\subseteq F$ and $Z\subseteq F^c$, since $(Y\setminus e)\cap (G\triangle H)=\emptyset$ and $Z\cap(G\triangle H)=\emptyset$.
\end{proof}
When $Y=\emptyset$ and $Z$ is a singleton, the previous lemma specializes to Lemma 2.10 in \cite{Chun_Moffatt_Noble_Rueckriemen_2019}; our proof is essentially an inductive application of their argument.

\begin{prop}\label{prop:leaf_recursion_for_delta-matroids}
    Let $D=(E,\mathcal{F})$ be a normal delta-matroid with distinct $e_1,e_2\in E$ such that $\{e_1,e_2\}\in\mathcal{F}$ and 
    $$D\setminus e_2=(\{e_1\},\{\emptyset\})\oplus D\setminus\{e_1,e_2\}.$$
    Then 
    $$\T{D}(z)=\T{D\setminus e_1}(z)+2z^2 \T{D\setminus\{e_1,e_2\}}(z).$$
\end{prop}
\begin{proof}
    Consider
    \[
    \tau_1=\{B\subseteq E:B\text{ contains exactly one of }e_1,e_2\}
    \]
    and
    \[
    \tau_2=\{B\subseteq E:B\text{ contains both or neither }e_1,e_2\}.
    \]
    Then  $\tau_1$ and $\tau_2$ partition the subsets of $E$, so
    $$\T{D}(z)=\sum_{B\in \tau_1}z^{w(D*B)}+\sum_{B\in\tau_2}z^{w(D*B)}.$$
Define the bijection $\sigma_1:\mathcal{P}(E(D\setminus e_1))\to\tau_1$ by
$$\sigma_1(A)=\begin{cases}
    A & \text{if } e_2\in A \\
    A\cup e_1 & \text{if }e_2\not\in A
\end{cases}$$
for $A\subseteq E(D\setminus e_1)$. Then if $e_2\in A$, by \autoref{lemma:width_of_twist_of_normal_delta_matroid},
$$w(D*\sigma_1(A))=w(D|_{\sigma_1(A)})+w(D|_{\sigma_1(A)^c})=w((D\setminus e_1)|_A)+w(D|_{A^c\cup e_1})$$
and $w(D|_{A^c\cup e_1})=w(D|_{A^c})=w((D\setminus e_1)|_{A^c})$ since $e_2\not\in A^c$ implies that $e_1$ is a loop in $D|_{A^c\cup e_1}$. Thus $w(D*\sigma_1(A))=w((D\setminus e_1)*A)$. The case $e_2\not\in A$ is symmetric. Hence
\[
\sum_{B\in \tau_1}z^{w(D*B)}=\sum_{A\subseteq E(D\setminus e_1)}z^{w(D*\sigma_1(A))}=\sum_{A\subseteq E(D\setminus e_1)}z^{w((D\setminus e_1)*A)}=\T{D\setminus e_1}(z).
\]
Next, define the injection $\sigma_2:\mathcal{P}(E(D\setminus\{e_1,e_2\}))\to\tau_2$ by $\sigma_2(A)=A\cup\{e_1,e_2\}$ for $A\subseteq E(D\setminus\{e_1,e_2\})$. Then $w(D|_{{\sigma_2(A)}^c})=w((D\setminus\{e_1,e_2\})|_{A^c})$ and by \autoref{lemma:wdith_of_restriction}, 
\[
w(D|_{\sigma_2(A)})=|{\sigma_2(A)}|-\min\{|{\sigma_2(A)}\triangle F|:F\in \mathcal{F}\}.
\]
On the other hand,
$w((D\setminus\{e_1,e_2\})|_A)=|A|-\min\{|A\triangle F|:F\in\mathcal{F}\text{ and }e_1,e_2\not\in F\}$. We show that
\[
m_1\coloneq \min\{|A\triangle F|:F\in\mathcal{F}\text{ and }e_1,e_2\not\in F\}=\min\{|{\sigma_2(A)}\triangle F|:F\in \mathcal{F}\}\eqcolon m_2.
\]
Observe that for all $F\in\mathcal{F}$, if $e_2\not\in F$, then $e_1\not\in F$. Now, if $F\in\mathcal{F}$ and $e_1,e_2\not\in F$, then there exists $u\in F\triangle\{e_1,e_2\}$ such that $F\triangle\{e_1,u\}\in\mathcal{F}$. Moreover, $e_1\in F\triangle\{e_1,u\}$, so it must be that $e_2\in F\triangle\{e_1,u\}$, so $u=e_2$. Thus, $F\triangle\{e_1,e_2\}\in\mathcal{F}$, and
\[
|A\triangle F|=|{\sigma_2(A)}\triangle(F\triangle\{e_1,e_2\})|,
\]
so $m_1\geq m_2$.
On the other hand, by \autoref{lemma:symmetric_difference_distance}, there exists $F\in\mathcal{F}$ such that $e_1,e_2\in F$ and $|{\sigma_2(A)}\triangle F|=m_2$. Then there exists $u\in F\triangle\emptyset=F$ such that $F\triangle\{e_2,u\}\in\mathcal{F}$. But $e_2\not\in F\triangle\{e_2,u\}$, so $e_1\not\in F\triangle\{e_2,u\}$, which means $u=e_1$. Thus, $e_1,e_2\not\in F\triangle\{e_1,e_2\}\in\mathcal{F}$, and
\[
|{\sigma_2(A)}\triangle F|=|A\triangle(F\triangle\{e_1,e_2\})|,
\]
so $m_1\leq m_2$. Hence,
\[
w(D|_{\sigma_2(A)})=w((D\setminus\{e_1,e_2\})|_A)+2,
\]
so
\begin{align*}
    \sum_{B\in \tau_2}z^{w(D*B)}&=2\sum_{A\subseteq E(D\setminus \{e_1,e_2\})}z^{w(D*\sigma_2(A))}=2\sum_{A\subseteq E(D\setminus \{e_1,e_2\})}z^{w((D\setminus \{e_1,e_2\})*A)+2}=2z^2\T{D\setminus \{e_1,e_2\}}(z),
\end{align*}
where the first equality holds because for every $B\in\tau_2$, exactly one of $B$ and $B^c$ is in the image of $\sigma_2$, and $w(D*B)=w(D*B^c)$.
\end{proof}
Applying \autoref{prop:leaf_recursion_for_delta-matroids} to $D(I(B))$ for a bouquet $B$ recovers \autoref{prop:Yan_Jin_recurrence}.

\section{Concluding Remarks}
We have shown that the twist polynomial satisfies a recursion formula for the one-point join of two looped simple graphs (i.e., normal, binary delta-matroids). It would be interesting to investigate whether the same recursion formula holds for a broader class of delta-matroids than the class of binary delta-matroids. \autoref{prop:leaf_recursion_for_delta-matroids} is a positive result in this direction. An important consideration would be to analyze when the one-point join of set systems given in \autoref{def:set_system_one-point_join} is a delta-matroid. The formula in \autoref{thm:looped_recursion} involves the loop complementation operation on graphs, which was generalized to set systems by Brijder and Hoogeboom in \cite{Brijder_Hoogeboom_2011}. However, the class of delta-matroids is not closed under loop complementation (see Example 10 in \cite{Brijder_Hoogeboom_2011}). Thus, any extension of \autoref{thm:looped_recursion} may have to be constrained within the class of vf-safe delta-matroids, introduced in \cite{Brijder_Hoogeboom_2013} as the class of delta-matroids that remain delta-matroids under any sequence of twists and loop complementations; they properly contain the class of binary delta-matroids. 

Finally, we remark that one can derive the results in Section \ref{sec:The Twist Polynomial of a One-Point Join of Graphs} for the special case of signed intersection graphs using purely topological considerations of the boundary components of bouquets instead of dealing with matrices. In particular, the partitioning into $P_1$ and $P_2$ in \autoref{lemma:recursion_existence} has a simple topological interpretation. This was the initial approach taken by the author, and the progress there illuminated the generalization to looped simple graphs.

\vspace{3mm}
\noindent
\textbf{Acknowledgments}\\

The main results in this work were achieved during the summer 2024 Knots and Graphs undergraduate research program at Ohio State University, which was supported by the NSF-DMS \#2231565 RTG grant: Arithmetic, Combinatorics, and Topology of Algebraic Varieties. The author is grateful to Sergei Chmutov for organizing the program and for helpful conversations.

\bibliography{mybib}
\bibliographystyle{plain}

\end{document}